\documentclass[12pt, two]{article}
\usepackage{}

\usepackage{CJK}
\usepackage{amsthm}
\usepackage{amsmath,amssymb,mathrsfs}
\usepackage{amsfonts}
\usepackage{yhmath}
\usepackage{graphics}
\usepackage{amssymb}
\usepackage{epsfig}
\usepackage{latexsym}
\usepackage{indentfirst}
\allowdisplaybreaks
\usepackage[top=1in, bottom=1in, left=1.25in, right=1.25in]{geometry}
\numberwithin{equation}{section}
\usepackage{CJK}

\newtheorem{theorem}{Theorem}[section]
\newtheorem{lemma}{Lemma}[section]

\newtheorem{corollary}{Corollary}[section]
\newtheorem{remark}{Remark}[section]

\numberwithin{equation}{section} %¹«Ê½Ëæ½Ú±àºÅ

\title{
Gradient estimate and Liouville theorems for $p$-harmonic maps
\footnotetext[0]{ $^{*}$Supported by NSFC grant No. 11771087.}
\footnotetext[0]{ ${}^{\dag}$ Supported by NSFC grant No. 11831005; the Natural Science Foundation of Fujian Province, China (Grant No.2017J01398).}
}
\author{
Yuxin Dong$^*$ and
  Hezi Lin$^\dag$ }
\date{}
\begin{document}
%\begin{CJK*}{GBK}{kai}

\maketitle

\begin{abstract}
In this paper, we first obtain an $L^q$ gradient estimate for $p$-harmonic maps, by assuming the target manifold supporting a certain function, whose gradient and Hessian satisfy some analysis conditions. From this $L^q$ gradient estimate, we get a corresponding Liouville type result for $p$-harmonic maps. Secondly, using these general results, we give various geometric applications to $p$-harmonic maps from complete manifolds with nonnegative Ricci curvature to manifolds with various upper bound on sectional curvature, under appropriate controlled images.
\end{abstract}

\section{Introduction}
It is well known that harmonic maps theory plays an important role in the geometric analysis and obtain vast research  on it. A natural generalization of harmonic maps  is $p$-harmonic maps  from
a variational point of view. $P$-harmonic maps arise in many different contexts in physics and mathematics,
and  it has been also extensively studied.

Recall that a $C^1$-map $u: (M^m, g)\rightarrow (N^n, h)$ between Riemannian manifolds is said to be $p$-harmonic, $p>1$, if it satisfies the non-linear system
\begin{equation}
\delta(|du|^{p-2} du)=0. \label{p-harmonic}
\end{equation}
The $p$-Laplace equation (\ref{p-harmonic}) arises as the Eucler-Lagrange equation associated with the $p$-energy functional $E_p(u)=\int_M |du|^p$. For more information on $p$-harmonic maps, one can consult the survey paper \cite{We}.

Fix $x \in  M$ with $|du|(x)\neq 0$. Choose a local orthonormal frame field $\{e_i\}_{i=1}^m$ at $x$. Denote by $f=|d u|$, then
\begin{align*}
\delta(f^{p-2}du) = \underset{i}{\sum}(\nabla_{e_i}(f^{p-2}du))(e_i)
=&\underset{i}{\sum}f^{p-2}(\nabla_{e_{i}}du)(e_{i})+ \underset{i}{\sum}e_{i}(f^{p-2})du(e_{i})\\
=& \underset{i}{\sum}f^{p-2}(\nabla_{e_{i}}du)(e_{i})+(p-2)f^{n-3}du(\nabla f).
\end{align*}
Hence, by (\ref{p-harmonic}) we have
\begin{equation}
\mbox{tr}_g (\nabla du) = -(p-2)f^{-1} du(\nabla f).  \label{Hessian-f}
\end{equation}

When $p=2$, $p$-harmonic maps are simply harmonic maps. In 1975, Yau \cite{Ya}
proved that on a complete Riemannian manifold with nonnegative Ricci curvature, any harmonic function bounded from one side must be a constant.   In \cite{SY}, Schoen and Yau showed that a harmonic map
of finite energy from a complete Riemannian manifold with non-negative
Ricci curvature to a complete manifold with non-positive sectional curvature is constant.  Later, Cheng \cite{Ch} obtained a gradient estimate of the
harmonic map to prove  the constancy of harmonic map if the target manifold
has nonpositive sectional curvature and the image satisfies a sublinear growth
condition. When the sectional curvature of the  target manifold has positive upper bound, Choi
\cite{Cho}  proved a Liouville theorem if the domain manifold has nonnegative
Ricci curvature and the image is contained in the geodesic ball.
When the target manifold $N$ is of sectional curvature bounded from above
by $-a^2$ for some $a>0$, it was proved by Y. Shen \cite{Sh} that if $u$ is a harmonic
map from a complete Riemannian manifold with non-negative
Ricci curvature into $N$ such that the image of $u$ lies in a horoball of $N$, then
$u$ must be a constant. There are also various Liouville type theorems for harmonic maps and their generalizations (e.g., \cite{DLY}, \cite{DW}, \cite{EL}, \cite{Hi}, \cite{Ji}  and the references therein).

For general $p>1$, compared with the theory for harmonic maps the study of $p$-harmonic maps is generally harder,
because the corresponding equation  is degenerate elliptic and the regularity results are far weaker. Recently, there has been renewed interest in $p$-harmonic maps. In particular, Moser \cite{M} established a nice connection between $p$-harmonic functions and the inverse mean curvature flow. In \cite{KN},  Kotschwar and Ni derived a local gradient estimate for $p$-harmonic functions by assuming the sectional curvature is bounded from below. Later, Wang and Zhang introduced the Moser iteration to improve Kotschwar-Ni's gradient estimate, by assuming only the Ricci curvature is bounded from below. Recently, Sung and  Wang \cite{SW} derive a sharp gradient estimate for positive $p$-harmonic functions on a complete manifold with Ricci curvature bounded below. For more Liouville type results for $p$-harmonic maps under  finite $q$-energy condition or pointwise energy condition,  one can consult \cite{CCW}, \cite{CS}, \cite{Ho}, \cite{Ma}, \cite{Na}, \cite{NT}, \cite{PRS}, \cite{PST}, \cite{PV}, \cite{Ta}, \cite{Zh} and the references therein.

In this paper, we tend to get  gradient estimate of general $p$-harmonic maps. However, the approach of Kotschwar-Ni's gradient estimate for positive $p$-harmonic functions is difficult to work in this case, since the linearization of the $p$-Laplace operator on maps involves vector bundle valued 1-forms which can not make inner product with real valued 1-forms in general.

To overcome this difficulty, we use the initial Bochner-Weitzenb\"{o}ck formula for $p$-harmonic maps to get the following $L^k$ gradient estimate.
\begin{theorem} \label{thm-main}
Suppose $(M^m, g)$ is a complete Riemannian manifold with $\mbox{Ric}^M \geq -K$, $K\geq 0$, and $(N^n, h)$ is a complete Riemannian manifold with $\mbox{Sec}^N \leq A$. Let $u : M \rightarrow N$ be a  $p$-harmonic map. Suppose there exists a positive $C^2$ function $\phi$ on $u(B_{x_0}(R))$  such that $A\underset{ u(B_{x_0}(R))}{\max}\phi < +\infty$, $|\nabla \phi| \leq c_0$ for some $c_0> 0$, and
\begin{equation*}
\frac{\alpha}{1+\phi}d \phi \otimes d\phi- \mbox{Hess} (\phi)\geq \beta h
\end{equation*}
for some positive constants $\alpha$, $\beta$, in the sense of quadratic forms.
Then there exists positive constants $\bar{k}_0(m, p, \sqrt{K}R)$, $C=C(m, p, k,\sqrt{K}R)$ and $C_1(m, p)$ such that
\begin{equation*}
\left|\frac{|du|^2}{C +\phi\circ u }\right|_{L^k(B_{x_0}(\frac{2R}{3}))}  \leq C_1 \frac{(1+\sqrt{K}R)^2}{R^2}(V(B_{x_0}(R))^{\frac{1}{k}}  \label{ineq-estimate}
\end{equation*}
holds for  $k \geq \bar{k}_0$.
\end{theorem}

If $A\leq 0$, then the condition $A\underset{ u(B_{x_0}(R))}{\max}\phi < +\infty$ is naturally holds for any positive function $\phi$. The novelty of our approach is that the geometric constraint of the target manifolds is now replaced by the assumption that there exists a positive function satisfying some analytic conditions.

Finally, we gives some applications of the above gradient estimate, and get Liouville type results for $p$-harmonic maps by assuming appropriate image constraints on the target manifolds according to the upper bound of the sectional curvature, which generalizes the corresponding results for  harmonic maps.
In order to get these Liouville properties for $p$-harmonic maps, it is natural to assume that the domain manifold has nonnegative Ricci curvature.

\section{Gradient estimate and the proof of Theorem 1.1}
Let us recall the following Bochner-Weitzenb\"{o}ck formula for $p$-harmonic maps, which can be deduced by applying formula (1.34) in \cite{EL} with the choice $\omega=|du|^{p-2}du$ (see also \cite{PRS}, \cite{CS}).
\begin{lemma}
Let $u: (M^m, g_M)\rightarrow (N^n, g_N)$ be any $p$-harmonic maps between Riemannian manifolds of dimensions $\mbox{dim} M =m$ and $\mbox{dim} N =n$, respectively. Suppose that
\begin{align*}
\mbox{Ricci}_x^M\geq -K(x) \ \ \mbox{on} \ M
\end{align*}
in the sense of quadratic forms,
\begin{align*}
\mbox{Sec}_{u(x)}^N \leq A(x) \ \ \mbox{on} \ N
\end{align*}
for some $A(x) $, then
\begin{align}
\nonumber |du|^{p-1}\triangle |du|^{p-1} \geq& - \langle \delta d (|du|^{p-2} du), |du|^{p-2} du\rangle\\
 &- \max\{0,A(x)\}|du|^{2p} -K(x)|du|^{2(p-1)}, \label{BW}
\end{align}
pointwise on the open set  $\{ x \in M: |du|(x)\neq 0\}$  and weakly on all of $M$.
\end{lemma}

Let  $(M^m, g)$ be a complete Riemannian manifold with $\mbox{Ricci}^M \geq - K $ and $\Omega \subset M$ be an open set. Let $(N^n, h)$ be a complete Riemannian manifold with $\mbox{Sec}^N \leq  A $ and $u: M\rightarrow N$ be a $p$-harmonic map.
Suppose there exists  a positive $C^2$ function $\widetilde{\phi}$ on the image $u(\Omega)$  such that
$|\nabla \widetilde{\phi}|\leq 2c_0$
for $c_0\geq 0$, and
\begin{equation*}
\frac{\tilde{\alpha}}{1+\widetilde{\phi}}d \widetilde{\phi}\otimes d\widetilde{\phi}- Hess (\widetilde{\phi})\geq \tilde{\beta} h
\end{equation*}
for some positive constants $\tilde{\alpha}$ and $\tilde{\beta}$. Let $\phi=C^2+ \widetilde{\phi}$, where $C$ is a positive constant to be determined later. Then
\begin{equation}
\phi\geq C^2, \ \ \  \ \ \ |\nabla \phi|\leq 2c_0,  \label{gra-phi}
\end{equation}
 and
\begin{equation}
\frac{\alpha}{\phi}d \phi \otimes d\phi- Hess (\phi)\geq \beta h \label{Hessian-phi}
\end{equation}
 for some positive constants $\alpha$ and $\beta$.

Denote by  $\varphi=|d u|^{p-1}$.
From (\ref{BW}), we get
\begin{align}
 \varphi \triangle  \varphi + K \varphi^2 \geq - \langle \delta d (|du|^{p-2} du), |du|^{p-2} du\rangle
 - \tilde{A}|d u|^{2p}, \label{BW1}
\end{align}
where $\tilde{A}= \max\{0, A\}$. Using (\ref{BW1}), we have
\begin{align}
\nonumber \varphi \triangle  \left(\frac{\varphi}{(\phi \circ u)^s} \right) + K \frac{\varphi^2}{(\phi \circ u)^s} \geq& - \frac{1}{(\phi \circ u)^s}\langle \delta d (|du|^{p-2} du), |du|^{p-2} du\rangle\\
\nonumber & - 2s\frac{\varphi \langle \nabla \varphi, \nabla (\phi \circ u)\rangle}{(\phi \circ u)^{s+1}} + s(s+1)\frac{\varphi^2|\nabla (\phi \circ u)|^2}{(\phi \circ u)^{s+2}} \\
 &- s\frac{\varphi^2 \triangle (\phi \circ u)}{(\phi \circ u)^{s+1}}- \tilde{A}\frac{|d u|^{2p}}{(\phi \circ u)^s}, \label{BW2}
\end{align}
where $s$ is a positive constant which will be determined later.
The relation (\ref{BW2}) holds wherever $|d u|$ is strictly positive. Let $Z=\{x \in \Omega: |d u|=0\}$. Then for any non-negative function $\psi$ with compact support in
$\Omega\setminus Z$, we have
\begin{align}
\nonumber &-\int_\Omega \left\langle \nabla \left(\frac{\varphi}{(\phi \circ u)^s} \right), \nabla (\varphi\psi) \right\rangle + K \int_\Omega \frac{\psi \varphi^2}{(\phi \circ u)^s}\\
\nonumber \geq& -\int_\Omega \frac{\psi}{(\phi \circ u)^s}\langle \delta d(|du|^{p-2}du), |du|^{p-2} du\rangle- 2s\int_\Omega\frac{\psi \varphi \langle \nabla \varphi, \nabla (\phi \circ u)\rangle}{(\phi \circ u)^{s+1}}\\
  &+ s(s+1)\int_\Omega \frac{\psi\varphi^2|\nabla (\phi \circ u)|^2}{(\phi \circ u)^{s+2}}
 - s\int_\Omega \frac{\psi\varphi^2 \triangle (\phi \circ u)}{(\phi \circ u)^{s+1}}
 -\tilde{A} \int_\Omega \frac{\psi |d u|^{2p}}{(\phi \circ u)^s}. \label{BW-int}
\end{align}
Denote by $f=:|d u|$, then $\varphi=f^{p-1}$. Let  $\psi =\frac{ f_\epsilon^q\eta^2}{(\phi \circ u)^t}$ where $f_\epsilon =(f-\epsilon)^{+}$ with $\epsilon >0$, $\eta \in C_0^{\infty}(\Omega)$ is nonnegative, $q >1$, $t\geq 0$ is to be determined later. Substituting this $\psi$ into the relation (\ref{BW-int}) yields
\begin{align}
\nonumber &-\int_\Omega \left\langle \nabla \left(\frac{f^{p-1}}{(\phi \circ u)^s} \right),
\nabla \left(\frac{f^{p-1}f_\epsilon^q\eta^2}{(\phi \circ u)^t}\right) \right\rangle + K \int_\Omega \frac{f^{2p-2}f_\epsilon^q\eta^2 }{(\phi \circ u)^{s+t}}\\
\nonumber  \geq& - \int_\Omega \left\langle  d (f^{p-2} du),d \left( \frac{f^{p-2}f_\epsilon^q\eta^2}{(\phi \circ u)^{s+t}} du \right) \right\rangle
  -2s\int_\Omega\frac{f^{p-1}f_\epsilon^q\eta^2 \langle \nabla f^{p-1}, \nabla (\phi \circ u)\rangle}{(\phi \circ u)^{s+t+1}}\\
  &+ s(s+1)\int_\Omega \frac{f^{2p-2}f_\epsilon^q\eta^2|\nabla (\phi \circ u)|^2}{(\phi \circ u)^{s+t+2}}
  -s\int_\Omega \frac{f^{2p-2}f_\epsilon^q\eta^2 \triangle (\phi \circ u)}{(\phi \circ u)^{s+t+1}}
  - \tilde{A} \int_\Omega \frac{f^{2p}f_\epsilon^q\eta^2 }{(\phi \circ u)^{s+t}}. \label{BW-int2}
\end{align}
We shall take care of the integrals in (\ref{BW-int2}) separately.  Let us first consider the left hand side of (\ref{BW-int2}).
\begin{align}
\nonumber & \mbox{LHS} (\ref{BW-int2})\\
\nonumber& = -(p-1)^2\int_\Omega \frac{f^{2p-4}f_\epsilon^q\eta^2}{(\phi \circ u)^{s+t}} |\nabla f|^2
-q(p-1)\int_\Omega \frac{f^{2p-3}f_\epsilon^{q-1}\eta^2}{(\phi \circ u)^{s+t}} |\nabla f|^2\\
\nonumber &-2(p-1)\int_\Omega \frac{f^{2p-3}f_\epsilon^{q}\eta}{(\phi \circ u)^{s+t}} \langle \nabla f, \nabla \eta\rangle +(s+t)(p-1)\int_\Omega \frac{f^{2p-3}f_\epsilon^{q}\eta^2}{(\phi \circ u)^{s+t+1}} \langle \nabla (\phi \circ u), \nabla f \rangle\\
\nonumber & + sq\int_\Omega \frac{f^{2p-2}f_\epsilon^{q-1}\eta^2}{(\phi \circ u)^{s+t+1}} \langle \nabla (\phi \circ u), \nabla f \rangle + 2s\int_\Omega \frac{f^{2p-2}f_\epsilon^{q}\eta}{(\phi \circ u)^{s+t+1}} \langle \nabla (\phi \circ u), \nabla \eta \rangle\\
 & -st\int_\Omega \frac{f^{2p-2}f_\epsilon^{q}\eta^2}{(\phi \circ u)^{s+t+2}} | \nabla (\phi \circ u)|^2
+ K \int_\Omega \frac{f^{2p-2}f_\epsilon^q\eta^2 }{(\phi \circ u)^{s+t}}. \label{LHS}
\end{align}
%where we have used the fact that $$|\nabla(\phi \circ u)|  \leq  | d \rho^2| | du| = 2 \rho |d \rho| | du| = 2\rho f.$$
To estimate the right hand side of (\ref{BW-int2}),  from Lemma 13 of \cite{PRS} we have
$|df\wedge du| \leq |df||du|$. Thus,
\begin{equation}
|d(f^{p-2}du)| = |d(f^{p-2})\wedge du | \leq |d(f^{p-2})|| du | =|p-2| f^{p-2}|\nabla f|, \label{product1}
\end{equation}
and
\begin{align}
\nonumber \langle d(f^{p-2}du), d(f^q_\epsilon du)\rangle=&\langle d(f^{p-2})\wedge du, d(f^q_\epsilon)\wedge du)\rangle\\
\nonumber =& q(p-2) f^{p-3}f^{q-1}_\epsilon |d(|du|)\wedge du |^2\\
 \leq & q\max\{0, p-2\}f^{p-1}f^{q-1}_\epsilon|\nabla f|^2. \label{product2}
\end{align}
 Fix $x \in \Omega$ with $|du|(x)\neq 0$. Choose a local orthonormal frame field $\{e_i\}_{i=1}^m$ at $x$. By (\ref{Hessian-f}), (\ref{gra-phi}) and (\ref{Hessian-phi}) we have
\begin{align}
\nonumber \alpha\frac{|\nabla(\phi \circ u)|^2}{\phi \circ u}-\triangle (\phi \circ u)  = &
\frac{\alpha}{\phi \circ u}\underset{i}{\sum}d \phi\otimes d\phi(du(e_i), du(e_i))-\underset{i}{\sum}(\nabla d \phi)(du(e_i), du(e_i))\\
\nonumber & - d \phi (\mbox{tr}_g (\nabla d u))\\
\nonumber \geq& \beta |du|^2 +(p-2)d\phi (f^{-1} du(\nabla f))\\
\nonumber \geq & \beta f^2 -|p-2||\nabla \phi||\nabla f|\\
\geq& \beta f^2-\frac{2c_0|p-2|}{C}  \sqrt{\phi}|\nabla f|. \label{Hessian-a}
\end{align}
%Considering the nonnegativity of the third term at the right hand side of (\ref{BW-int2}), and
Using (\ref{product1}), (\ref{product2}), (\ref{Hessian-a}) and the Cauchy-Schwarz inequality, we calculate the right hand side of (\ref{BW-int2}),
\begin{align}
\nonumber  \mbox{RHS}& (\ref{BW-int2})\\
\nonumber =& -\int_\Omega  |d(f^{p-2}du)|^2 \frac{f_\epsilon^q\eta^2}{(\phi \circ u)^{s+t}}
    -\int_\Omega \frac{f^{p-2}\eta^2}{(\phi \circ u)^{s+t}} \left\langle d(f^{p-2}du), d\left( f_\epsilon^q du\right)  \right\rangle\\
\nonumber& -\int_\Omega f^{p-2}f_\epsilon^q \left\langle d(f^{p-2}du), d\left( \frac{\eta^2}{(\phi \circ u)^{s+t}} du\right)  \right\rangle\\
\nonumber  & - 2s\int_\Omega\frac{f^{p-1}f_\epsilon^q\eta^2 \langle \nabla f^{p-1}, \nabla (\phi \circ u)\rangle}{(\phi \circ u)^{s+t+1}} + s(s+1-\alpha)\int_\Omega \frac{f^{2p-2}f_\epsilon^q\eta^2|\nabla (\phi \circ u)|^2}{(\phi \circ u)^{s+t+2}}\\
\nonumber  & + s\int_\Omega \frac{f^{2p-2}f_\epsilon^q\eta^2 }{(\phi \circ u)^{s+t+1}}\left[ \alpha \frac{|\nabla(\phi \circ u)|^2}{\phi \circ u}-\triangle (\phi \circ u) \right]-\tilde{A} \int_\Omega \frac{f^{2p}f_\epsilon^q\eta^2 }{(\phi \circ u)^{s+t}}\\
\nonumber \geq& -(p-2)^2\int_\Omega   \frac{f^{2p-4}f_\epsilon^q\eta^2}{(\phi \circ u)^{s+t}}|\nabla f|^2-q\max\{0, p-2\}\int_\Omega \frac{f^{2p-3}f_\epsilon^{q-1} \eta^2}{(\phi \circ u)^{s+t}}|\nabla f|^2\\
 \nonumber  & -|p-2|\int_\Omega f^{2p-3}f_\epsilon^q |\nabla f| \left| d\left( \frac{\eta^2}{(\phi \circ u)^{s+t}} \right)\right| - 2s(p-1)\int_\Omega\frac{f^{2p-3}f_\epsilon^q\eta^2 \langle \nabla f, \nabla (\phi \circ u)\rangle}{(\phi \circ u)^{s+t+1}}\\
\nonumber  & + s(s+1-\alpha)\int_\Omega \frac{f^{2p-2}f_\epsilon^q\eta^2|\nabla (\phi \circ u)|^2}{(\phi \circ u)^{s+t+2}}+ s\beta \int_\Omega \frac{f^{2p}f_\epsilon^q\eta^2 }{(\phi \circ u)^{s+t+1}}\\
\nonumber  &-\frac{2sc_0 |p-2|}{C} \int_\Omega \frac{f^{2p-2}f_\epsilon^q \eta^2\sqrt{\phi}|\nabla f| }{(\phi \circ u)^{s+t+1}}-\tilde{A} \int_\Omega \frac{f^{2p}f_\epsilon^q\eta^2 }{(\phi \circ u)^{s+t}}\\
 \nonumber   \geq&  -\left\{(p-2)^2+ \frac{s^2 c_0 |p-2|}{C}\right\}\int_\Omega   \frac{f^{2p-4}f_\epsilon^q\eta^2}{(\phi \circ u)^{s+t}}|\nabla f|^2
    -q\max\{0, p-2\}\int_\Omega \frac{f^{2p-3}f_\epsilon^{q-1} \eta^2}{(\phi \circ u)^{s+t}}|\nabla f|^2 \\
\nonumber&   -2|p-2|\int_\Omega \frac{f^{2p-3}f_\epsilon^{q} \eta}{(\phi \circ u)^{s+t}}|\nabla f||\nabla \eta|-(s+t)|p-2|\int_\Omega \frac{f^{2p-3}f_\epsilon^{q} \eta^2}{(\phi \circ u)^{s+t+1}}|\nabla (\phi \circ u)||\nabla f|\\
  \nonumber  &- 2s(p-1)\int_\Omega\frac{f^{2p-3}f_\epsilon^q\eta^2 \langle \nabla f, \nabla (\phi \circ u)\rangle}{(\phi \circ u)^{s+t+1}} + s(s+1-\alpha)\int_\Omega \frac{f^{2p-2}f_\epsilon^q\eta^2|\nabla (\phi \circ u)|^2}{(\phi \circ u)^{s+t+2}}\\
  & + \left(s\beta-\frac{c_0|p-2|}{C}\right) \int_\Omega \frac{f^{2p}f_\epsilon^q\eta^2 }{(\phi \circ u)^{s+t+1}}-\tilde{A} \int_\Omega \frac{f^{2p}f_\epsilon^q\eta^2 }{(\phi \circ u)^{s+t}}. \label{RHS}
\end{align}
Combining with (\ref{LHS}) and (\ref{RHS}), denoting by $A(\phi)=\tilde{A}\underset{x \in \Omega}{\max}\phi \circ u$ and considering that $p-1 -\max\{0, p-2\}=\min\{1, p-1\} $, we have
\begin{align}
\nonumber  &\left(2p-3-\frac{s^2c_0 |p-2|}{C}\right)\int_\Omega   \frac{f^{2p-4}f_\epsilon^q\eta^2}{(\phi \circ u)^{s+t}}|\nabla f|^2
   + q\min\{1, p-1\}\int_\Omega \frac{f^{2p-3}f_\epsilon^{q-1} \eta^2}{(\phi \circ u)^{s+t}}|\nabla f|^2\\
\nonumber &+ s(s+t+1-\alpha)\int_\Omega \frac{f^{2p-2}f_\epsilon^q\eta^2|\nabla (\phi \circ u)|^2}{(\phi \circ u)^{s+t+2}}+ \left(s\beta-\frac{c_0 |p-2|}{C}-A(\phi)\right) \int_\Omega \frac{f^{2p}f_\epsilon^q\eta^2 }{(\phi \circ u)^{s+t+1}}\\
\nonumber & \leq sq\int_\Omega \frac{f^{2p-2}f_\epsilon^{q-1}\eta^2}{(\phi \circ u)^{s+t+1}} \langle \nabla (\phi \circ u), \nabla f \rangle
 +(3s+t)(p-1)\int_\Omega \frac{f^{2p-3}f_\epsilon^{q}\eta^2}{(\phi \circ u)^{s+t+1}} \langle \nabla (\phi \circ u), \nabla f \rangle\\
\nonumber &+ (s+t)|p-2|\int_\Omega \frac{f^{2p-3}f_\epsilon^{q} \eta^2}{(\phi \circ u)^{s+t+1}}|\nabla (\phi \circ u)||\nabla f|+ 2|p-2|\int_\Omega \frac{f^{2p-3}f_\epsilon^{q} \eta}{(\phi \circ u)^{s+t}}|\nabla f||\nabla \eta| \\
\nonumber &- 2(p-1)\int_\Omega \frac{f^{2p-3}f_\epsilon^{q}\eta}{(\phi \circ u)^{s+t}} \langle \nabla f, \nabla \eta\rangle+2s\int_\Omega \frac{f^{2p-2}f_\epsilon^{q}\eta}{(\phi \circ u)^{s+t+1}} \langle \nabla (\phi \circ u), \nabla \eta \rangle\\
 &+ K \int_\Omega \frac{f^{2p-2}f_\epsilon^q\eta^2 }{(\phi \circ u)^{s+t}}. \label{BW-int3}
\end{align}
Using both dominated and monotone convergence, taking the limit on both side of (\ref{BW-int3}) as $\epsilon \rightarrow 0$, we obtain
\begin{align}
\nonumber  &\left(2p-3+q \min\{1, p-1\}-\frac{s^2c_0 |p-2|}{C}\right)\int_\Omega   \frac{f^{2p+q-4}\eta^2}{(\phi \circ u)^{s+t}}|\nabla f|^2\\
\nonumber &+s(s+t+1-\alpha)\int_\Omega \frac{f^{2p+q-2}\eta^2|\nabla (\phi \circ u)|^2}{(\phi \circ u)^{s+t+2}}\\
\nonumber&+ \left(s\beta-\frac{c_0 |p-2|}{C}-A(\phi)\right) \int_\Omega \frac{f^{2p+q}\eta^2}{(\phi \circ u)^{s+t+1}}\\
\nonumber & \leq  \left[(s+t)(p-1) + s(2p+q-2) \right]\int_\Omega \frac{f^{2p+q-3}\eta^2}{(\phi \circ u)^{s+t+1}} \langle \nabla (\phi \circ u), \nabla f \rangle\\
\nonumber &  + (s+t)|p-2|\int_\Omega \frac{f^{2p+q-3} \eta^2}{(\phi \circ u)^{s+t+1}}|\nabla (\phi \circ u)||\nabla f|+2|p-2|\int_\Omega \frac{f^{2p+q-3} \eta}{(\phi \circ u)^{s+t}}|\nabla f||\nabla \eta| \\
\nonumber &-2(p-1)\int_\Omega \frac{f^{2p+q-3}\eta}{(\phi \circ u)^{s+t}} \langle \nabla f, \nabla \eta\rangle+ 2s\int_\Omega \frac{f^{2p+q-2}\eta}{(\phi \circ u)^{s+t+1}} \langle \nabla (\phi \circ u), \nabla \eta \rangle\\
\nonumber &+ K \int_\Omega \frac{f^{2p+q-2}\eta^2 }{(\phi \circ u)^{s+t}}\\
 \nonumber & \leq  \left[(s+t)(p-1+|p-2|) + s(2p+q-2) \right]\int_\Omega \frac{f^{2p+q-3} \eta^2}{(\phi \circ u)^{s+t+1}}|\nabla (\phi \circ u)||\nabla f|\\
\nonumber &+2(p-1+|p-2|)\int_\Omega \frac{f^{2p+q-3} \eta}{(\phi \circ u)^{s+t}}|\nabla f||\nabla\eta| \\
&+ 2s\int_\Omega \frac{f^{2p+q-2}\eta}{(\phi \circ u)^{s+t+1}} | \nabla (\phi \circ u)|| \nabla \eta|
 + K \int_\Omega \frac{f^{2p+q-2}\eta^2 }{(\phi \circ u)^{s+t}}. \label{BW-int4}
\end{align}
 For simplicity of notation and discussion, we only consider the case of $p\geq 2$. Choose $s$, $t$ such that $s+t = \frac{2p+q-2}{2}=p+\frac{q}{2}-1$. Denote by $k=p+\frac{q}{2}-1$. Then (\ref{BW-int4}) reduces to
\begin{align}
\nonumber  &\frac{(2k-1)C-s^2c_0 |p-2|}{k^2C}\int_\Omega  \left| \frac{\eta \nabla f^k}{(\phi \circ u)^{k/2}} \right|^2
+\frac{4s(k+1-\alpha)}{k^2}\int_\Omega \left| f^{k}\eta \nabla \left( \frac{1}{(\phi \circ u)^{k/2}} \right) \right|^2\\
\nonumber  & + \left(s\beta-\frac{c_0 |p-2|}{C}-A(\phi) \right) \int_\Omega \left(\frac{f^{2}}{\phi \circ u}\right)^{k+1}\eta^2\\
 \nonumber &\leq \frac{2(2s+p-1+|p-2|)}{k} \int_\Omega \left| \frac{\eta \nabla f^k}{(\phi \circ u)^{k/2}}\right| \left| \eta  f^{k}\nabla \left( \frac{1}{(\phi \circ u)^{k/2}} \right) \right|\\
 \nonumber & + \frac{2(p-1+|p-2|)}{k}\int_\Omega \left| \frac{\eta \nabla f^k}{(\phi \circ u)^{k/2}}\right| \left| \frac{f^{k}\nabla \eta}{(\phi \circ u)^{k/2}}\right|\\
 & + \frac{4s}{k} \int_\Omega \left|\frac{f^{k}\nabla \eta} {(\phi \circ u)^{k/2}}\right| \left|  \eta  f^{k}\nabla \left( \frac{1}{(\phi \circ u)^{k/2}} \right) \right| + K\int_\Omega \left(\frac{f^{2}}{\phi \circ u}\right)^k\eta^2.  \label{BW-int40}
\end{align}
Denote by $R_i$ (resp. $L_i$), $i=1,2,3, \cdots$, the i-th term at the right (resp. left) hand side.
Since $\phi \geq C^2$, we have
$$|\nabla(\phi \circ u)|^2  \leq  | \nabla \phi|^2 | du|^2 \leq \frac{4c_0^2}{C^2} |\phi| f^2.$$
Using the Cauchy inequality  to the right hand side  of (\ref{BW-int40}) yields
\begin{align}
\nonumber R_1 \leq&  \frac{3}{2k}\int_\Omega  \left| \frac{\eta \nabla f^k}{(\phi \circ u)^{k/2}} \right|^2 + \frac{k(2s+p-1+|p-2|)^2}{6}\int_\Omega \frac{f^{2k}\eta^2}{(\phi \circ u)^{k}} \left|\frac{\nabla (\phi \circ u)}{\phi \circ u}\right|^2\\
 \leq&  \frac{3}{2k}\int_\Omega  \left| \frac{\eta \nabla f^k}{(\phi \circ u)^{k/2}} \right|^2 + \frac{2kc_0^2(2s+p-1+|p-2|)^2}{3C^2}\int_\Omega \frac{f^{2k+2}\eta^2}{(\phi \circ u)^{k+1}}. \label{R1}
\end{align}
\begin{align}
 R_2 \leq  \frac{1}{4k}\int_\Omega  \left| \frac{\eta \nabla f^k}{(\phi \circ u)^{k/2}} \right|^2
+ \frac{4(p-1+|p-2|)^2}{k}\int_\Omega  \left| \frac{f^{k}\nabla \eta}{(\phi \circ u)^{k/2}}\right|^2.  \label{R2}
\end{align}
\begin{align}
 R_3 \leq \frac{2s}{k}  \int_\Omega \left|\frac{f^{k}\nabla \eta} {(\phi \circ u)^{k/2}}\right|^2  +\frac{2s}{k} \int_\Omega \left|  \eta  f^{k}\nabla \left( \frac{1}{(\phi \circ u)^{k/2}} \right) \right|^2.   \label{R3}
\end{align}
Substituting (\ref{R1}), (\ref{R2}) and (\ref{R3}) into (\ref{BW-int40}), we have
\begin{align}
\nonumber  &\frac{(k-4)C-4s^2c_0 |p-2|}{4k^2C}\int_\Omega  \left| \frac{\eta \nabla f^k}{(\phi \circ u)^{k/2}} \right|^2
+\frac{2s(k+2-2\alpha)}{k^2}\int_\Omega \left| f^{k}\eta \nabla \left(\frac{1}{(\phi \circ u)^{k/2}} \right) \right|^2 \\
\nonumber  & + \left(s\beta-\frac{c_0 |p-2|}{C}-A(\phi) \right) \int_\Omega \left(\frac{f^{2}}{\phi \circ u}\right)^{k+1}\eta^2\\
 \nonumber \leq& \frac{2kc_0^2(2s+p-1+|p-2|)^2}{3C^2}\int_\Omega \frac{f^{2k+2}\eta^2}{(\phi \circ u)^{k+1}}
+ \frac{2s+4(p-1+|p-2|)^2}{k}\int_\Omega \left|\frac{f^{k}\nabla \eta} {(\phi \circ u)^{k/2}}\right|^2\\
 &+ K\int_\Omega \left(\frac{f^{2}}{\phi \circ u}\right)^k\eta^2.  \label{BW-int5}
\end{align}
By choosing $C$ large enough such that
\begin{equation}
  \frac{2kc_0^2(2s+p-1+|p-2|)^2}{3C^2} \leq \frac{s\beta}{2},  \label{k-large}
\end{equation}
 the relation (\ref{BW-int5}) reduces to
\begin{align}
\nonumber  &\frac{(k-4)C-4s^2c_0 |p-2|}{4k^2C}\int_\Omega  \left| \frac{\eta \nabla f^k}{(\phi \circ u)^{k/2}} \right|^2
+\frac{2s(k+2-2\alpha)}{k^2}\int_\Omega \left| f^{k}\eta \nabla \left(\frac{1}{(\phi \circ u)^{k/2}} \right) \right|^2 \\
\nonumber  & +\left(\frac{s\beta}{2}-\frac{c_0 |p-2|}{C}-A(\phi)\right) \int_\Omega \left(\frac{f^{2}}{\phi \circ u}\right)^{k+1}\eta^2\\
 \leq& \frac{2s+4(p-1+|p-2|)^2}{k}\int_\Omega \left|\frac{f^{k}\nabla \eta} {(\phi \circ u)^{k/2}}\right|^2+ K\int_\Omega \left(\frac{f^{2}}{\phi \circ u}\right)^k\eta^2.  \label{BW-int6}
\end{align}
Assume that $A(\phi)=\tilde{A}\underset{x \in \Omega}{\max}\phi \circ u < \infty$. Choose $k$, $s$ large enough such that
\begin{equation}
(k-4)C-4s^2c_0 |p-2| >0, \  k+2-2\alpha>0 \ \mbox{and}\ \ \frac{s\beta}{2}-\frac{c_0 |p-2|}{C}-A(\phi) >0. \label{positive}
\end{equation}
Since,
\begin{align*}
  \int_\Omega   \left|\nabla( \frac{ f^k\eta}{(\phi \circ u)^{k/2}})\right|^2
  \leq &   3\int_\Omega  \left| \frac{\eta \nabla f^k}{(\phi \circ u)^{k/2}} \right|^2 + 3 \int_\Omega \left| f^{k}\eta \nabla \left(\frac{1}{(\phi \circ u)^{k/2}} \right) \right|^2\\
  &+ 3\int_\Omega \left|\frac{f^{k}\nabla \eta} {(\phi \circ u)^{k/2}}\right|^2,
\end{align*}
combining  this with (\ref{BW-int6}) gives
\begin{align}
\int_\Omega   |\nabla( \Phi^{\frac{k}{2}}\eta)|^2 + ka_1\left(\frac{s\beta}{2}-A(\phi)\right) \int_\Omega \Phi^{k+1}\eta^2
 \leq a_2\int_\Omega \Phi^k \left|\nabla\eta\right|^2+ kKa_3\int_\Omega \Phi^k\eta^2,  \label{BW-int7}
\end{align}
where we denote by $\Phi=\frac{f^{2}}{\phi \circ u} $ and $a_i$, $i=1,2,\cdots$ are the constants depending only on $m$ and $p$.

Let us recall the following Sobolev embedding theorem of Saloff-Coste.
For $m>2$, there exists $C_0$, depending only on $m$, such that for any ball $B \subset M$ of radius
$R$ and volume $V$, we have
\begin{equation}
\left(\int_B |f|^{\frac{2m}{m-2}}  \right)^{\frac{m-2}{m}} \leq e^{C_0(1+\sqrt{K}R)}V^{-\frac{2}{m}}
 R^2 \int_B \left( |\nabla f|^2 +  R^{-2} f^2 \right)   \label{sob-int}
\end{equation}
for any $f \in C^{\infty}_0$. For $m=2$, the above inequality holds with $m$ replaced by any fixed $m' >2$.

Let $B_{x_0}( R)$ be the closed geodesic  ball of radius $R$ and center $x_0$ in $M$. Let $\Omega = B_{x_0}( R)$. Applying (\ref{sob-int}) we have
\begin{equation}
\left(\int_\Omega \Phi^{\frac{mk}{m-2}} \eta^{\frac{2m}{m-2}} \right)^{\frac{m-2}{m}} \leq e^{C_0(1+\sqrt{K}R)}V^{-\frac{2}{m}}
 \left( R^2 \int_\Omega |\nabla ( \Phi^{\frac{k}{2}}\eta)|^2 + \int_\Omega \Phi^k \eta^2 \right). \label{sob-inq2}
\end{equation}
Let $k_0 =c_1(n,p)(1+\sqrt{K}R)$ with $c_1(n,p)$ large enough such that (\ref{positive}) holds for $k=k_0$. Combining (\ref{BW-int7}) and (\ref{sob-inq2}) gives
\begin{align}
\nonumber  &\left(\int_\Omega \Phi^{\frac{mk}{m-2}} \eta^{\frac{2m}{m-2}} \right)^{\frac{m-2}{m}}
   + a_4k(s\beta/2-A(\phi)) e^{k_0}V^{-\frac{2}{m}}R^2 \int_\Omega \Phi^{k+1}\eta^2\\
  \leq&    a_5  k_0^2 k e^{k_0}V^{-\frac{2}{m}} \int_\Omega \Phi^k\eta^2 + a_6 e^{k_0}V^{-\frac{2}{m}}R^2 \int_\Omega  \Phi^k|\nabla \eta|^2.  \label{main-int}
\end{align}

Now we can prove our main result.
\begin{theorem} \label{thm-pr}
Suppose $(M^m, g)$ is a complete Riemannian manifold with $\mbox{Ric}^M \geq -K$, $K\geq 0$, and $(N^n, h)$ is a complete Riemannian manifold with $\mbox{Sec}^N \leq A$. Let $u : M \rightarrow N$ be a  $p$-harmonic map. Suppose there exists a positive $C^2$ function $\widetilde{\phi}$ on $u(B_{x_0}(R))$ such that $ A\underset{u(B_{x_0}(R))}{\max}\widetilde{\phi} < \infty$, $|\nabla \widetilde{\phi}| \leq 2c_0$ for some $c_0\geq 0$, and
\begin{equation}
\frac{\alpha}{1+\widetilde{\phi}}d \widetilde{\phi} \otimes d\widetilde{\phi}- \mbox{Hess} (\widetilde{\phi})\geq \beta h  \label{hess-ine}
\end{equation}
for some positive constants $\alpha$, $\beta$. Then there exists positive constants $\bar{k}_0(m, p, \sqrt{K}R)$, $C=C(m, p, k,\sqrt{K}R)$ and $C_1(m, p)$ such that
\begin{equation}
\left|\frac{|du|^2}{C +\widetilde{\phi}\circ u }\right|_{L^k(B(\frac{2R}{3}))}  \leq C_1(m, p) \frac{(1+\sqrt{K}R)^2}{R^2}(V(B(R))^{\frac{1}{k}}  \label{ineq-estimate}
\end{equation}
holds for  $k \geq \bar{k}_0$.
\end{theorem}
\begin{proof}
Let $k=k_0$ in (\ref{main-int}). To compare $L_2$ and $R_1$ of (\ref{main-int}), we observe that if $\Phi > \frac{2a_5 k_0^2}{a_4(s\beta/2-A(\widetilde{\phi})) R^2}$, then
\begin{equation*}
a_5 k_0^3 \Phi^k < \frac{1}{2}a_4 k_0 (s\beta/2-A(\widetilde{\phi})) R^2 \Phi^{k+1}.
\end{equation*}
Hence we decompose $\Omega$ as
\begin{align*}
\Omega= \left\{x\in \Omega | \Phi(x) > \frac{2a_5 k_0^2}{a_4(s\beta/2-A(\widetilde{\phi})) R^2}\right \}\bigcup \left\{x\in \Omega | \Phi(x) \leq \frac{2a_5 k_0^2}{a_4(s\beta/2-A(\widetilde{\phi})) R^2} \right\}.
\end{align*}
With this decomposition we have
\begin{equation*}
R_1 \leq a_7^{k_0} k^3_0  \left(\frac{k_0}{R} \right)^{2k_0} e^{k_0}V^{1-\frac{2}{m}} + \frac{1}{2}L_2.
\end{equation*}
Therefore (\ref{main-int}) with $k=k_0$ can be written as
\begin{align}
\nonumber  &\left(\int_\Omega \Phi^{\frac{mk_0}{m-2}} \eta^{\frac{2m}{m-2}} \right)^{\frac{m-2}{m}}
   + \frac{1}{2}a_4(s\beta/2-A(\widetilde{\phi}))k_0 e^{k_0}V^{-\frac{2}{m}}R^2 \int_\Omega \Phi^{k_0+1}\eta^2\\
  \leq&    a_7^{k_0} k^3_0  \left(\frac{k_0}{R} \right)^{2k_0} e^{k_0}V^{1-\frac{2}{m}}  + a_6 e^{k_0}V^{-\frac{2}{m}}R^2 \int_\Omega  \Phi^{k_0}|\nabla \eta|^2.   \label{main-int2}
\end{align}
Now we choose $\eta$ to make $R_2$ of (\ref{main-int2}) dominated by the LHS. Choose  $\eta_1 \in C^{\infty}_0(M)$ satisfying
\begin{equation*}
0 \leq \eta_1 \leq 1, \ \ \eta_1\equiv 1 \ \mbox{in}\ B_{x_0}(2R/3), \ \ |\nabla \eta_1|\leq \frac{C(m)}{R}.
\end{equation*}
Let $\eta = \eta_1^{k_0 +1 }$. Then
\begin{equation*}
R^2 |\nabla \eta|^2 \leq  a_8 k_0^2 \eta^{\frac{2k_0}{k_0+1 }}.
\end{equation*}
Hence applying the Young's inequality shows that
\begin{align}
\nonumber a_6R^2 \int_\Omega  \Phi^{k_0}|\nabla \eta|^2  \leq& a_9 k_0^2  \int_\Omega  \Phi^{k_0}\eta^{\frac{2k_0}{k_0+1 }}\\
\nonumber\leq & a_9 k_0^2\left(\int_\Omega \Phi^{k_0+1}\eta^2 \right)^{\frac{k_0}{k_0 +1 }} V^{\frac{1}{k_0 +1 }}\\
 \leq & \frac{1}{2}a_4 \left(s\beta/2-A(\widetilde{\phi})\right)k_0R^2 \int_\Omega \Phi^{k_0 +1}\eta^2 + a_{10}^{k_0} \frac{k_0^{2k_0 +1}}{R^{2k_0}}V. \label{ineq2}
\end{align}
Substituting (\ref{ineq2}) into (\ref{main-int2})  yields
\begin{align}
  \left(\int_{B_{x_0}(2R/3)} \Phi^{\frac{mk_0}{m-2}}\right)^{\frac{m-2}{m}}
  \leq    a_{11}^{k_0} k^3_0 e^{k_0} \left(\frac{k_0}{R} \right)^{2k_0} V^{1-\frac{2}{m}}.  \label{main-int3}
\end{align}
Setting $k_1=\frac{mk_0}{m-2} $, taking the $\frac{1}{k_0}$-root on both sides of (\ref{main-int3})
we have
\begin{equation*}
\|\Phi\|_{L^{k_1}(B_{x_0}(2R/3))} \leq a_{11}ek_0^{\frac{3}{k_0}} \frac{k_0^2}{R^2}V^{\frac{1}{k_1}}
\leq a_{11}e^{1+\frac{3}{e}}\frac{k_0^2}{R^2}V^{\frac{1}{k_1}},
\end{equation*}
which completes the proof of Theorem \ref{thm-pr}.
\end{proof}

\begin{remark}
Under the assumption $\underset{u(B_{x_0}(R))}{\max}\widetilde{\phi} < \infty$, the condition (\ref{hess-ine}) is equivalent to
$\mbox{Hess} (\log(\widetilde{\phi}+1)) \leq -\alpha h$ for some  $\alpha>0$.
To eliminate the last term in (\ref{R1}), we introduce the constant $C$. But according to (\ref{k-large}), the constant $C$ in (\ref{ineq-estimate}) depends on $k$, so we can not further proceed the Moser iteration to the inequality (\ref{main-int}) to get local gradient estimate.
\end{remark}

\section{Liouville type theorems}
Based on the above estimate in Theorem \ref{thm-pr}, by assuming the Ricci curvature of the domain manifold is nonnegative and using the volume comparison theorem, we have the following Liouville type theorem.
\begin{theorem} \label{Liouville-thm}
Suppose $(M^m, g)$ is a complete Riemannian manifold with $\mbox{Ric}^M \geq 0$ and $(N^n, h)$ is a complete Riemannian manifold with $\mbox{Sec}^N \leq A$. Let $u : M \rightarrow N$ be a  $p$-harmonic map. Suppose there exists a positive $C^2$ function $\phi$ on $u(M)$ such that $A\underset{ u(M)}{\max}\phi  < \infty$, $|\nabla \phi| \leq c_0$ for some $c_0> 0$, and $\frac{\alpha}{1+\phi}d \phi \otimes d\phi- \mbox{Hess} (\phi)\geq \beta h$
for some positive constants $\alpha$ and $\beta$. Then $u$ is a constant map.
\end{theorem}
\begin{proof}
Since $\mbox{Ric}^M \geq 0$, by the volume comparison theorem we have
\begin{equation*}
V(B_{x_0}(R)) \leq C(m)R^m.
\end{equation*}
Letting $k>\frac{m}{2}$ in (\ref{ineq-estimate}), we obtain
\begin{equation*}
\left|\frac{|du|^2}{C +\phi\circ u }\right|_{L^k(M)}  \leq \underset{R\rightarrow \infty}{\lim}C_1(m,p) \frac{(V(B_{x_0}(R)))^{\frac{1}{k}}}{R^2}\leq \underset{R\rightarrow \infty}{\lim}C_2(m,p)\frac{R^{\frac{m}{k}}}{R^2}=0,
\end{equation*}
which implies that $du=0$, that is, $u$ is a constant map.
\end{proof}

%Using the $L^k$ gradient estimate of Theorem \ref{thm-pr},%
 By choosing various concrete $C^2$ functions $\phi$ satisfying (\ref{gra-phi}) and (\ref{Hessian-phi}) on the target manifold $N$ according to the upper bound of the sectional curvature, we can further get some Liouville type theorems for $p$-harmonic maps
under appropriate constraints on the images of the maps.

If the target manifold has nonpositive sectional curvature, we have the following Liouville theorem for $p$-harmonic maps which generalizes the result in \cite{Ch}.
\begin{corollary}
Suppose $(M^m, g)$ is a complete Riemannian manifold with nonnegative Ricci curvature and $(N^n, h)$ is a complete Riemannian manifold  with nonpositive sectional curvature. Let $u : M \rightarrow N$ be a   $p$-harmonic map. If the image  $u(M)$ lies in a compact subset of $N$, then $u$ is a constant map.
\end{corollary}
\begin{proof}
For any $x_0 \in M$,
let   $\rho$ denote the distance function on $N$ from $u(x_0)$ and $\phi=a-\rho^2$,  where
$a > \sup_{x\in  M}\rho^2 (u(x))$.
A direct computation shows that, on the image $u(M)$,
\begin{align*}
|\nabla \phi|=|d \rho^2| =2\rho|d \rho|  \leq 2\sqrt{a}.
\end{align*}
By  the composition formula (cf. \cite{EL}) and the standard Hessian comparison theorem, we have
\begin{align*}
- \mbox{Hess} (\phi) =\mbox{Hess}(\rho^2)=& 2\rho \mbox{Hess}(\rho) + 2 d\rho\otimes d \rho\\
\geq& 2(h-d\rho\otimes d \rho ) +2 d\rho\otimes d \rho\\
=&2h.
\end{align*}
Therefore, Theorem \ref{Liouville-thm} immediately gives the proof.
\end{proof}

If the sectional curvature of the target manifold is bounded from above by a positive constant, The corresponding Liouville property can be stated as follows.
\begin{corollary}
Suppose $(M^m, g)$ is a complete Riemannian manifold with nonnegative Ricci curvature and $(N^n, h)$ is a complete Riemannian manifold with $\mbox{Sec}^N \leq A$, $A>0$. Let $u : M \rightarrow N$ be a  $p$-harmonic map, such that $u(M) \subset B_R(y_0)$. If   $B_R(y_0)$ lies inside the cut locus of $y_0$ and $R < \frac{\pi}{2\sqrt{A}}$. Then $u$ is a constant map.
\end{corollary}

\begin{proof}
For any $x_0 \in M$,
let   $\rho$ denote the distance function on $N$ from $u(x_0)$ and let $\phi = \cos (\sqrt{A} \rho)  $.
Since $R < \frac{\pi}{2\sqrt{A}}$, we get
\begin{equation*}
0<\cos (\sqrt{A} \rho) \leq 1 \ \ \ \ \mbox{and}\ \ A \underset{u(M)}{\max}\cos (\sqrt{A}\rho)\leq A.
\end{equation*}
Then
\begin{align*}
|\nabla \phi|=\sqrt{A} |\sin(\sqrt{A} \rho)|| d \rho| \leq \sqrt{A}.
\end{align*}
By   the Hessian comparison theorem, we have
\begin{align*}
-\mbox{Hess}   \cos (\sqrt{A} \rho)=& \sqrt{A}\sin (\sqrt{A} \rho)\mbox{Hess}(\rho) + A\cos(\sqrt{A}\rho) d\rho\otimes d \rho\\
\geq& A\sin(\sqrt{A}\rho) \cot(\sqrt{A} \rho)(h-d\rho\otimes d \rho)+ A\cos (\sqrt{A} \rho ) d\rho\otimes d \rho\\
\geq & A\cos(\sqrt{A}R) h.
\end{align*}
Therefore, the proof is completed by Theorem \ref{Liouville-thm}.
\end{proof}

Suppose $N$ is a complete
Riemannian manifold of nonpositive curvature and let $\gamma: \mathbb{R}\rightarrow N$ be a unit speed
geodesic. The union of balls $B_\gamma = \bigcup_{t>0}B_{\gamma(t)}(t)$ is called the horoball with center $c(+\infty)$. For $y \in N$, the function $t\rightarrow t- d(y, \gamma(t))$ is bounded from above and
monotonically increasing. So we can define the Busemann function by
\begin{equation*}
B(y)=\underset{t\rightarrow +\infty}{\lim}(t- d(y, \gamma(t))).
\end{equation*}
 It is easy to see that $B(y)>0$ on $B_\gamma$. We also know that $B(y)$ is a concave $C^2$-function with
$|\nabla B| =1$.

When the target manifold $N$ satisfies $\mbox{Sec}^N \leq -A$, $A>0$, we have
\begin{corollary}
Suppose $(M^m, g)$ is a complete Riemannian manifold with nonnegative Ricci curvature and $(N^n, h)$ is a complete Riemannian manifold with $\mbox{Sec}^N \leq -A$, $A>0$. Let $u : M \rightarrow N$ be a  $p$-harmonic map, such that $u(M) \subset B_\gamma$, where $B_\gamma$ is some horoball in $N$ centered at $\gamma(+\infty)$ with respect to a unit speed geodesic $\gamma(t)$. Then $u$ is a constant map.
\end{corollary}
\begin{proof}
Let $\phi(y) = B^2(y)$. We may assume that $B(y) \geq 1$ since Busemann functions are horofunctions. Then,
\begin{align*}
|\nabla \phi|=2 B|\nabla B|   \leq 2 B,
\end{align*}
which is bounded from above. Applying the Hessian comparison theorem, we have
\begin{align*}
-B_{\alpha \beta}=\underset{t\rightarrow \infty}{\lim}\rho_{\alpha \beta}(y, \gamma(t))  \geq& \sqrt{A}\coth(\sqrt{A}\rho(y, \gamma(t)))(h_{\alpha \beta} - dB_\alpha \otimes dB_\beta )\\
\geq& \sqrt{A}(h_{\alpha \beta} - dB_\alpha \otimes dB_\beta).
\end{align*}
Set $\alpha$ such that $4\alpha-2-2\sqrt{A} \geq 0$. Then
\begin{align*}
\frac{\alpha}{\phi}d \phi \otimes d\phi- Hess (\phi)  =& 4\alpha dB \otimes dB - 2B \mbox{Hess}(B) -2dB \otimes dB\\
\geq & (4\alpha-2)dB \otimes dB + 2 \sqrt{A}(h - dB \otimes dB)\\
\geq & 2 \sqrt{A}h.
\end{align*}
Therefore Theorem \ref{Liouville-thm} implies the proof.
\end{proof}

%\begin{corollary}
%Suppose $M$ is a complete Riemannian manifold with nonnegative Ricci curvature and $N$ is a simply-connected complete Riemannian manifold with $\mbox{Sec}^N \leq -A$, $A>0$.
 %Let $u : M \rightarrow N$ be a  $p$-harmonic map, such that $u(M) \subset B_\gamma$,
 %where $B_\gamma$ is some horoball in $N$ centered at $\gamma(+\infty)$ with respect to
 % a unit speed geodesic $\gamma(t)$. Then $u$ is a constant map.
%\end{corollary}

\begin{remark}
The negativity condition on the sectional curvature of the target manifold is necessary. For example, let
\begin{equation*}
u= (u_1, u_2):\mathbb{R}^2 \rightarrow \mathbb{R}^2.
\end{equation*}
If $u_1 =\mbox{constant}$ and $u_2$ be any nonconstant $p$-harmonic function, then $u$ is a nonconstant $p$-harmonic map.
In Euclidean space, the level sets of the Busemann function associated with a given geodesic are the planes perpendicular to the given geodesic. Hence, $u(\mathbb{R}^2)$ lies in a horoball.
\end{remark}

\vspace{0.5cm}
$\mathbf{Acknowledgements}$.  The authors would like to thank Prof. Guofang Wang for his interest and suggestions.

Yuxin Dong, School of Mathematical Sciences and
Laboratory of Mathematics for Nonlinear Science,
Fudan University, Shanghai, 200433, China.

E-mail address: yxdong@fudan.edu.cn\\

Hezi Lin, College of Mathematics and Informatics \&  FJKLMAA,
 Fujian Normal University, Fuzhou,  350108,
 China.

E-mail address: lhz1@fjnu.edu.cn\\

%\end{CJK*}
\end{document}